\documentclass[11pt]{article}
\usepackage{amsfonts}
\usepackage{amsmath}
\usepackage{amsthm}

\newtheorem{theorem}{Theorem}[section]

\newtheorem{proposition}{Proposition}[section]
\newtheorem{corollary}{Corollary}[section]

\theoremstyle{definition}
\newtheorem{definition}{Definition}[section]
\newtheorem{example}{Example}[section]
\input{tcilatex}

\begin{document}

\title{SURFACES GIVEN WITH THE MONGE PATCH IN $\mathbb{E}^{4}$}
\author{Bet\"{u}l Bulca and Kadri Arslan \\
Department of Mathematics, Uluda\u{g} University ,16059 Bursa, TURKEY \\
E-mails: bbulca@uludag.edu.tr, arslan@uludag.edu.tr}
\maketitle

\begin{abstract}
A depth surface of $\mathbb{E}^{3}$ is a range image observed from a single
view can be represented by a digital graph (Monge patch) surface . That is,
a depth or range value at a point $(u,v)$ is given by a single valued
function $z=f(u,v)$. In the present study we consider the surfaces in
Euclidean 4-space $\mathbb{E}^{4}$ given with a Monge patch $%
z=f(u,v),w=g(u,v)$. We investigated the curvature properties of these
surfaces. We also give some special examples of these surfaces which are
first defined by Yu. Aminov. Finally, we proved that every Aminov surface is
a non-trivial Chen surface.
\end{abstract}

\section{\textbf{Introduction}}

\footnote{%
2010 \textit{Mathematics Subject Classification}. 53C40, 53C42
\par
\textit{Key words and phrases}: Monge patch, Translation surface, Chen
surface} In recent years there has been a tremendous increase in computer
vision research using range images (or depth maps) as sensor input data \cite%
{BJ}. The most attractive feature of range images is the explicitness of the
surface information. Many industrial and navigational robotic tasks will be
more easily accomplished if such explicit depth information can be
efficiently obtained and interpreted. Classical differential geometry
provides a complete local description of smooth surfaces \cite{Ca}, \cite{Li}%
. The first and second fundamental forms of surfaces provide a set of
differential-geometric shape descriptors that capture domain-independent
surface information. Gaussian curvature is an intrinsic surface property
which refers to an isometric invariant of a surface \cite{Ca}. Both Gaussian
and mean curvatures have the attractive characteristics of translational and
rotational invariance. A depth surface is a range image observed from a
single view can be represented by a digital graph (Monge patch) surface.
That is, a depth or range value at a point $(u,v)$ is given by a single
valued function $z(u,v)$.

One interesting class of surfaces in $\mathbb{E}^{3}$ is that of translation
surfaces, which can be parameterized, locally, as $z(u,v)=f(u)+g(v)$, where $%
f$ and $g$ are smooth functions. From the definition, it is clear that
translation surfaces are double curved surfaces. Therefore, translation
surfaces are made up of quadrilateral, that is, four sided, facets. Because
of this property, translation surfaces are used in architecture to design
and construct free-form glass roofing structures, see \cite{GSC}. Scherk's
surface, obtained by H. Scherk in 1835, is the only non flat minimal
surface, that can be represented as a translation surface \cite{Sc}.
Translation surfaces have been investigated from the various viewpoints by
many differential geometers. L. Verstraelen, J. Walrave and S. Yaprak have
investigated minimal translation surfaces in n-dimensional Euclidean spaces 
\cite{VWY}.

In \cite{Ch} B.Y. Chen defined the allied vector field $a(v)$ of a normal
vector field $v$. In particular, the allied mean curvature vector field is
orthogonal to $H$. Further, B.Y. Chen defined the $\mathcal{A}$-surface to
be the surfaces for which $a(H)$ vanishes identically. Such surfaces are
also called Chen surfaces \cite{GVV1}. The class of Chen surfaces contains
all minimal and pseudo-umbilical surfaces, and also all surfaces for which $%
dimN_{1}\leq 1,$ in particular all hypersurfaces. These Chen surfaces are
said to be Trivial $\mathcal{A}$-surfaces \cite{GVV2}. In \cite{Ro}, B.
Rouxel considered ruled Chen surfaces in $\mathbb{E}^{n}.$ For more details,
see also, \cite{Du} and \cite{IAG}.

This paper is organized as follows: Section $2$ gives some basic concepts of
the surfaces in $\mathbb{E}^{4}$. Section $3$ tells about the surfaces given
with a Monge patch in $\mathbb{E}^{4}$. Further this section provides some
basic properties of surfaces in $\mathbb{E}^{4}$ and the structure of their
curvatures. In the third section we consider Aminov surfaces given with the
Monge patch in $\mathbb{E}^{4}$. We also present some examples of these
surfaces. We obtain few new interesting results. Namely, we obtain some
equations on $r(u),$ when on $M$ the equation $K+K_{N}=0$ has place.\ Then
we obtain the condition for the case $M$ is a Wintgen ideal surface. We
remark that on the Wintgen ideal surfaces the equation $K+K_{N}=\left\Vert
H\right\Vert ^{2}$ has place. In the final section we obtain an important
equation on the coefficents of the second quadratic form for Chen surfaces,
when it is given at arbitrary parametrization. We also proved that every
Aminov surfaces in $\mathbb{E}^{4}$ are non-trivial Chen surfaces.\ 

\section{\textbf{Basic Concepts}}

Let $M$ be a smooth surface in $\mathbb{E}^{4}$ given with the patch $X(u,v)$
: $(u,v)\in D\subset \mathbb{E}^{2}$. The tangent space to $M$ at an
arbitrary point $p=X(u,v)$ of $M$ span $\left\{ X_{u},X_{v}\right\} $. In
the chart $(u,v)$ the coefficients of the first fundamental form of $M$ are
given by 
\begin{equation}
E=\left\langle X_{u},X_{u}\right\rangle ,F=\left\langle
X_{u},X_{v}\right\rangle ,G=\left\langle X_{v},X_{v}\right\rangle ,
\label{A1}
\end{equation}%
where $\left\langle ,\right\rangle $ is the Euclidean inner product. We
assume that $W^{2}=EG-F^{2}\neq 0,$ i.e. the surface patch $X(u,v)$ is
regular.\ For each $p\in M$, consider the decomposition $T_{p}\mathbb{E}%
^{4}=T_{p}M\oplus T_{p}^{\perp }M$ where $T_{p}^{\perp }M$ is the orthogonal
component of $T_{p}M$ in $\mathbb{E}^{4}$. Let $\overset{\sim }{\nabla }$ be
the Riemannian connection of $\mathbb{E}^{4}$. Given any local vector fields 
$X_{i},$ $X_{j}$ tangent to $M$.

Let $\chi (M)$ and $\chi ^{\perp }(M)$ be the space of the smooth vector
fields tangent to $M$ and the space of the smooth vector fields normal to $M$%
, respectively. Consider the second fundamental map: $h:\chi (M)\times \chi
(M)\rightarrow \chi ^{\perp }(M);$%
\begin{equation}
h(X_{i},X_{j})=\widetilde{\nabla }_{X_{i}}X_{j}-\nabla _{X_{i}}X_{j}\text{\ }%
1\leq i,j\leq 2.  \label{A3}
\end{equation}%
where $\widetilde{\nabla }$ is the induced. This map is well-defined,
symmetric and bilinear.

For any arbitrary orthonormal normal frame field $\left\{
N_{1},N_{2}\right\} $ of $M$, recall the shape operator $A:\chi ^{\perp
}(M)\times \chi (M)\rightarrow \chi (M);$%
\begin{equation}
A_{N_{i}}X_{i}=-(\widetilde{\nabla }_{X_{i}}N_{i})^{T},\text{ \ \ }X_{i}\in
\chi (M).  \label{A4}
\end{equation}%
This operator is bilinear, self-adjoint and satisfies the following equation:%
\begin{equation}
\left\langle A_{N_{k}}X_{i},X_{j}\right\rangle =\left\langle
h(X_{i},X_{j}),N_{k}\right\rangle =c_{ij}^{k}\text{, }1\leq i,j,k\leq 2.
\label{A5}
\end{equation}

The equation (\ref{A3}) is called Gaussian formula, and%
\begin{equation}
h(X_{i},X_{j})=\overset{2}{\underset{k=1}{\sum }}c_{ij}^{k}N_{k},\ \ \ \ \
1\leq i,j\leq 2  \label{A6}
\end{equation}%
where $c_{ij}^{k}$ are the coefficients of the second fundamental form.

Further, the \textit{Gaussian curvature} and \textit{Gaussian torsion} of a
regular patch $X(u,v)$ are given by%
\begin{equation}
K=\frac{1}{W^{2}}\sum%
\limits_{k=1}^{2}(c_{11}^{k}c_{22}^{k}-(c_{12}^{k})^{2}),  \label{A7}
\end{equation}%
and 
\begin{equation}
K_{N}=\frac{1}{W^{2}}\left( E\left(
c_{12}^{1}c_{22}^{2}-c_{12}^{2}c_{22}^{1}\right) -F\left(
c_{11}^{1}c_{22}^{1}-c_{11}^{2}c_{22}^{1}\right) +G\left(
c_{11}^{1}c_{12}^{2}-c_{11}^{2}c_{12}^{1}\right) \right) ,  \label{A8}
\end{equation}%
respectively.

Further, the mean curvature vector of a regular patch $X(u,v)$ is defined by%
\begin{equation}
\overrightarrow{H}=\frac{1}{2W^{2}}%
\sum_{k=1}^{2}(c_{11}^{k}G+c_{22}^{k}E-2c_{12}^{k}F)N_{k}.  \label{A9}
\end{equation}

Recall that a surface $M$ is said to be \textit{minimal} if its mean
curvature vector vanishes identically \cite{Ch}.

The surface patch $X(u,v)$ is called \textit{pseudo-umbilical} if the shape
operator with respect to $H$ is proportional to the identity (see, \cite{Ch}%
).

\section{Surfaces Given with a Monge Patch in $\mathbb{E}^{4}$}

2-dimensional surfaces \ in $\mathbb{E}^{4}$ are interesting object for
investigation of geometers. Here we have some difficult problems which wait
its solutions. For example, it is unknown does there exist an isometric
regular immersion of the whole Lobachevsky plane into $\mathbb{E}^{4}$. \
Hence the investigation of various classes of surfaces in $\mathbb{E}^{4}$
with point of view of influence of the principal invariants - Gauss
curvature $K$, Gauss torsion $K_{N}$ and the vector of mean curvature $%
\overrightarrow{H}$ on the behavior of surfaces is an actual problem.

In the considering work we use the representation of surfaces in the
explicit form 
\begin{equation}
r(u,v)=(u,v,f(u,v),g(u,v)),  \label{C1}
\end{equation}%
where $f$ and $g$ are some smooth functions.{\Large \ }The parametrization (%
\ref{C1}) is called \textit{Monge patch in }$\mathbb{E}^{4}.$

First we obtain the following result.

\begin{theorem}
Let $M$ be a smooth surface given with the Monge patch (\ref{C1}). Then the
mean curvature vector of $M$ becomes%
\begin{eqnarray}
\overrightarrow{H} &=&\frac{1}{2\sqrt{A}W^{2}}(Gf_{uu}-2Ff_{uv}+Ef_{vv})N_{1}
\label{C2} \\
&&+\frac{1}{2\sqrt{A}W^{3}}\left(
G(-Bf_{uu}+Ag_{uu})-2F(-Bf_{uv}+Ag_{uv})+E(-Bf_{vv}+Ag_{vv})\right) N_{2} 
\notag
\end{eqnarray}%
where 
\begin{eqnarray}
A &=&1+(f_{u})^{2}+(f_{v})^{2},  \notag \\
B &=&f_{u}g_{u}+f_{v}g_{v}\text{ },  \label{C3} \\
C &=&1+(g_{u})^{2}+(g_{v})^{2}  \notag
\end{eqnarray}%
such that $EG-F^{2}=AC-B^{2}.$
\end{theorem}

\begin{proof}
The tangent space of $M$ is spanned by the vector fields%
\begin{eqnarray*}
\frac{\partial X}{\partial u} &=&(1,0,f_{u},g_{u}), \\
\frac{\partial X}{\partial v} &=&(0,1,f_{v},g_{v}).
\end{eqnarray*}%
Hence, the coefficients of the first fundamental form of the surface are%
\begin{eqnarray}
E &=&\text{ }\left\langle X_{u}(u,v),X_{u}(u,v)\right\rangle \text{ }%
=1+(f_{u})^{2}+(g_{u})^{2},  \notag \\
F &=&\text{ }\left\langle X_{u}(u,v),X_{v}(u,v)\right\rangle \text{ }%
=f_{u}f_{v}+g_{u}g_{v},  \label{C4} \\
G &=&\text{ }\left\langle X_{v}(u,v),X_{v}(u,v)\right\rangle \text{ }%
=1+(f_{v})^{2}+(g_{v})^{2},  \notag
\end{eqnarray}%
where $\left\langle ,\right\rangle $ is the standard scalar product in $%
\mathbb{R}^{4}.$

The second partial derivatives of $X(u,v)$ are expressed as follows%
\begin{eqnarray}
X_{uu}(u,v) &=&(0,0,f_{uu},g_{uu}),  \notag \\
X_{uv}(u,v) &=&(0,0,f_{uv},g_{uv}),  \label{C5} \\
X_{vv}(u,v) &=&(0,0,f_{vv},g_{vv}).  \notag
\end{eqnarray}

Further, the normal space of $M$ is spanned by the vector fields%
\begin{eqnarray}
N_{1} &=&\frac{1}{\sqrt{A}}(-f_{u},-f_{v},1,0)  \label{C6} \\
N_{2} &=&\frac{1}{W\sqrt{A}}(Bf_{u}-Ag_{u},Bf_{v}-Ag_{v},-B,A).
\notag
\end{eqnarray}

Using (\ref{A5}), (\ref{C5}) and (\ref{C6}) we can calculate the
coefficients of the second fundamental form $h$ are as follows:%
\begin{eqnarray}
c_{11}^{1} &=&\left\langle X_{uu}(u,v),N_{1}\right\rangle =\frac{f_{uu}}{%
\sqrt{A}},\text{ \ \ \ }  \notag \\
\text{\ }c_{12}^{1} &=&\left\langle X_{uv}(u,v),N_{1}\right\rangle =\frac{%
f_{uv}}{\sqrt{A}},  \notag \\
\text{ }c_{22}^{1} &=&\left\langle X_{vv}(u,v),N_{1}\right\rangle =\frac{%
f_{vv}}{\sqrt{A}},\text{ }  \label{C7} \\
c_{11}^{2} &=&\left\langle X_{uu}(u,v),N_{2}\right\rangle =\frac{%
-Bf_{uu}+Ag_{uu}}{W\sqrt{A}},  \notag \\
\text{ }c_{12}^{2} &=&\left\langle X_{uv}(u,v),N_{2}\right\rangle =\frac{%
-Bf_{uv}+Ag_{uv}}{W\sqrt{A}},\text{ }  \notag \\
c_{22}^{2} &=&\left\langle X_{vv}(u,v),N_{2}\right\rangle =\frac{%
-Bf_{vv}+Ag_{vv}}{W\sqrt{A}}.  \notag
\end{eqnarray}

Further, substituting (\ref{C4}) and (\ref{C7}) into (\ref{A9}) we get (\ref%
{C2}). This completes the proof of the theorem.
\end{proof}

In \cite{Am} Yu. Aminov proved the following result.

\begin{theorem}
\cite{Am} Let $M$ be a smooth surface given with the Monge patch (\ref{C1}).
Then the Gaussian curvature $K$ and Gaussian torsion $K_{N}$ of $M$ become%
\begin{equation}
K=\frac{%
C(f_{uu}f_{vv}-f_{uv}^{2})-B(f_{uu}g_{vv}+g_{uu}f_{vv}-2f_{uv}g_{uv})+A(g_{uu}g_{vv}-g_{uv}^{2})%
}{W^{4}},  \label{C8}
\end{equation}%
and%
\begin{equation}
K_{N}=\frac{%
E(f_{uv}g_{vv}-g_{uv}f_{vv})-F(f_{uu}g_{vv}-g_{uu}f_{vv})+G(f_{uu}g_{uv}-g_{uu}f_{uv})%
}{W^{4}}  \label{C9}
\end{equation}%
respectively.
\end{theorem}

\begin{proposition}
Let $M$ be a smooth surface given with the Monge patch of the form 
\begin{eqnarray}
f(u,v) &=&\phi _{u}(u,v),  \label{C10} \\
g(u,v) &=&\phi _{v}(u,v).  \notag
\end{eqnarray}%
Then the Gaussian curvature $K$ coincides with the Gaussian torsion $K_{N}$
of $M.$
\end{proposition}

\begin{proof}
Suppose $M$ is a smooth surface given with the Monge patch
(\ref{C1}). \
Then by the use of (\ref{C3}) with (\ref{C4}) we get%
\begin{eqnarray}
E &=&A=1+(\phi _{uu})^{2}+(\phi _{uv})^{2},  \notag \\
F &=&B=\phi _{uu}\phi _{uv}+\phi _{uv}\phi _{vv}\text{ },  \label{C11} \\
G &=&C=1+(\phi _{uv})^{2}+(\phi _{vv})^{2}  \notag
\end{eqnarray}%
Furthermore, substituting (\ref{C10}) into (\ref{C8})-(\ref{C9}) and
using partial derivatives of the functions given in the equation
(\ref{C10}) we obtain $K=K_{N}.$
\end{proof}

\begin{example}
For the surface $M$ given with the Monge patch 
\begin{eqnarray*}
f(u,v) &=&\phi _{u}(u,v)=e^{u}\cos v. \\
g(u,v) &=&\phi _{v}(u,v)=-e^{u}\sin v
\end{eqnarray*}%
the Gaussian curvature $K$ coincides with the Gaussian torsion $K_{N}$ of $M$
\cite{Am2}. $.$
\end{example}

\begin{definition}
The surface given with the parametrization (\ref{C1}) by the parametrization%
\begin{equation}
f(u,v)=f_{3}(u)+g_{3}(v),\text{ }g\text{ }(u,v)=f_{4}(u)+g_{4}(v)
\label{C12}
\end{equation}%
is called \textit{translation surface} in Euclidean 4-space $\mathbb{E}^{4}$ 
\cite{DVVZ}$.$
\end{definition}

In the case (\ref{C12}) we obtain simple expressions for $K,K_{N}$ and $%
\overrightarrow{H}$. As a consequence of Theorem 1 and Theorem 2 we get the
following results.

\begin{corollary}
Let $M$ be a translation surface given with the Monge patch (\ref{C12}).
Then the Gaussian curvature $K$ \ and Gaussian torsion $K_{N}$ of $M$ becomes%
\begin{equation*}
K=\frac{f_{3}^{\prime \prime }(u)g_{3}^{\prime \prime }(v)C-(f_{3}^{\prime
\prime }(u)g_{4}^{\prime \prime }(v)+f_{4}^{\prime \prime }(u)g_{3}^{\prime
\prime }(v))B+f_{4}^{\prime \prime }(u)g_{4}^{\prime \prime }(v)A}{W^{4}},
\end{equation*}%
and%
\begin{equation*}
K_{N}=\frac{F(f_{4}^{\prime \prime }(u)g_{3}^{\prime \prime
}(v)-f_{3}^{\prime \prime }(u)g_{4}^{\prime \prime }(v))}{W^{4}},
\end{equation*}%
respectively, where%
\begin{eqnarray*}
E &=&\text{ }1+(f_{3}^{\prime }(u))^{2}+(f_{4}^{\prime }(u))^{2}, \\
F &=&\text{ }f_{3}^{\prime }(u)g_{3}^{\prime }(v)+f_{4}^{\prime
}(u)g_{4}^{\prime }(v), \\
G &=&1+(g_{3}^{\prime }(v))^{2}+(g_{4}^{\prime }(v))^{2},
\end{eqnarray*}%
and%
\begin{eqnarray*}
A &=&\text{ }1+(f_{3}^{\prime }(u))^{2}+(g_{3}^{\prime }(v))^{2}, \\
B &=&\text{ }f_{3}^{\prime }(u)f_{4}^{\prime }(u)+g_{3}^{\prime
}(v)g_{4}^{\prime }(v), \\
C &=&1+(f_{4}^{\prime }(u))^{2}+(g_{4}^{\prime }(v))^{2}.
\end{eqnarray*}
\end{corollary}

\begin{corollary}
Let $M$ be a translation surface given with the Monge patch (\ref{C12}).
Then the mean curvature vector of $M$ becomes%
\begin{equation*}
\overrightarrow{H}=\frac{f_{3}^{\prime \prime }(u)G+g_{3}^{\prime \prime
}(v)E}{2\sqrt{A}W^{2}}N_{1}+\frac{G(f_{4}^{\prime \prime }(u)A-f_{3}^{\prime
\prime }(u)B)+E(g_{4}^{\prime \prime }(v)A-g_{3}^{\prime \prime }(v)B)}{2%
\sqrt{A}W^{3}}N_{2}.
\end{equation*}
\end{corollary}

\begin{example}
The translation surface given with the surface patch of%
\begin{equation*}
X(u,v)=(u,v,u^{2}+v^{2},u^{2}-v^{2})
\end{equation*}%
has vanishing Gaussian curvature and Gaussian torsion \cite{Am}$.$
\end{example}

\begin{theorem}
\cite{DVVZ} Let $M$ be a translation surface in $\mathbb{E}^{4}$. Then $M$
is minimal if and only if either $M$ is a plane or%
\begin{eqnarray*}
f_{k}(u) &=&\frac{c_{k}}{c_{3}^{2}+c_{4}^{2}}\left( \log \left\vert \cos (%
\sqrt{a}u)\right\vert +cu\right) +e_{k}u, \\
g_{k}(v) &=&\frac{c_{k}}{c_{3}^{2}+c_{4}^{2}}\left( -\log \left\vert \cos (%
\sqrt{b}v)\right\vert +dv\right) +p_{k}v,\text{ }k=3,4,
\end{eqnarray*}%
where $c_{k}$, $e_{k}$, $p_{k}$, $a>0,b>0,c,d$ are real constants.
\end{theorem}

\section{Aminov Surfaces in $\mathbb{E}^{4}$}

In the present section we consider the surfaces $M$ with 
\begin{equation}
f(u,v)=r(u)\cos v,\text{ }g\text{ }(u,v)=r(u)\sin v.  \label{D1}
\end{equation}%
which earlier were been considering in the work \cite{Am2}. We call such
surfaces \textit{Aminov surfaces} in Euclidean 4-space $\mathbb{E}^{4}.$

As a consequence of \ Theorem 2 we get the following result.

\begin{corollary}
Let $M$ be an Aminov surface given with the Monge patch (\ref{D1}). Then the
Gaussian curvature $K$ and Gaussian torsion $K_{N}$ of $M$ becomes%
\begin{equation}
K=-\frac{r(u)r^{\prime \prime }(u)(1+r^{2}(u))+(r^{\prime
}(u))^{2}(1+(r^{\prime }(u))^{2})}{(1+r^{2}(u))^{2}(1+(r^{\prime
}(u))^{2})^{2}},  \label{D2}
\end{equation}%
and%
\begin{equation}
K_{N}=\frac{r^{\prime }(u)r^{\prime \prime }(u)(1+r^{2}(u))+r(u)r^{\prime
}(u)(1+(r^{\prime }(u))^{2})}{(1+r^{2}(u))^{2}(1+(r^{\prime }(u))^{2})^{2}}
\label{D3}
\end{equation}%
respectively.
\end{corollary}

\begin{proposition}
Let $M$ be an Aminov surface given with the Monge patch (\ref{D1}). If $%
K+K_{N}=0,$ then the equality%
\begin{equation}
(r(u)-(r^{\prime }(u))\left( (r^{\prime }(u)(1+(r^{\prime
}(u))^{2})-r^{\prime \prime }(u)(1+r^{2}(u))\right) =0  \label{D6}
\end{equation}%
holds.
\end{proposition}

\begin{proof}
Using (\ref{D2}) and (\ref{D3}) we get the result.
\end{proof}

As a consequence of Proposition 8 we can give the following example.

\begin{example}
The Aminov surface given with the surface patch of%
\begin{equation}
X(u,v)=(u,v,\lambda e^{u}\cos v,\lambda e^{u}\sin v).  \label{D7}
\end{equation}%
satisfies the relation $K+K_{N}=0.$
\end{example}

As a consequence of Theorem 1 we get the following results.

\begin{proposition}
Let $M$ be an Aminov surface given with the Monge patch (\ref{D1}). Then the
mean curvature vector of $M$ becomes%
\begin{equation}
\overrightarrow{H}=\frac{\left( Gr^{\prime \prime }(u)-Er(u)\right) }{2W^{2}%
\sqrt{A}}\left\{ \cos vN_{1}+\left( \frac{A\sin v-B\cos v}{W}\right)
N_{2}\right\} .  \label{D8}
\end{equation}%
where%
\begin{eqnarray*}
A &=&1+(r^{\prime }(u))^{2}\cos ^{2}v+r^{2}(u)\sin ^{2}v, \\
B &=&\left( (r^{\prime }(u))^{2}-r^{2}(u)\right) \cos v\sin v, \\
C &=&1+(r^{\prime }(u))^{2}\sin ^{2}v+r^{2}(u)\cos ^{2}v.
\end{eqnarray*}%
and%
\begin{eqnarray*}
E &=&1+(r^{\prime }(u))^{2}, \\
F &=&0, \\
G &=&1+r^{2}(u).
\end{eqnarray*}%
such that $EG-F^{2}=AC-B^{2}.$
\end{proposition}

\begin{corollary}
Let $M$ be an Aminov surface given with the Monge patch (\ref{D1}). Then the
mean curvature of $M$ becomes%
\begin{equation}
H=\frac{r^{\prime \prime }(u)(1+r^{2}(u))-r(u)(1+(r^{\prime }(u))^{2})}{%
2(1+r^{2}(u))(1+(r^{\prime }(u))^{2})^{3/2}}  \label{D9}
\end{equation}
\end{corollary}

\begin{corollary}
Let $M$ be an Aminov surface given with the Monge patch (\ref{D1}). If $M$
is minimal then 
\begin{equation}
r(u)=\frac{1}{2a}\left( a^{2}e^{\pm \frac{2(u+b)}{a}}+a^{2}-1\right) e^{\pm 
\frac{(u+b)}{a}},  \label{D10}
\end{equation}%
where, $a$ and $b$ are real constants.
\end{corollary}

\begin{proof}
Suppose that $M$ is minimal then using the equality (\ref{D9}) we
get
\begin{equation}
r^{\prime \prime }(u)(1+r^{2}(u))-r(u)(1+(r^{\prime }(u))^{2})=0.
\label{D11}
\end{equation}%
Further, by the use of Maple and easy calculation shows that
(\ref{D10})\ is a non-trivial solution of (\ref{D11}).
\end{proof}

\begin{definition}
A surface $M$ is said to be \textit{Wintgen ideal surface} in $\mathbb{E}%
^{4} $ if \ the equality%
\begin{equation}
K+\left\vert K_{N}\right\vert =\left\Vert \overrightarrow{H}\right\Vert ^{2}
\label{D12}
\end{equation}%
holds \cite{W}.
\end{definition}

We obtain the following result.

\begin{theorem}
Let $M$ be an Aminov surface given with the Monge patch (\ref{D1}). If $M$
is \textit{Wintgen ideal surface} then the equality%
\begin{equation}
2r^{\prime \prime }(1+r^{2})(1+(r^{\prime })^{2})(2r^{\prime
}-r)+(1+(r^{\prime })^{2})^{2}(4rr^{\prime }-4(r^{\prime
})^{2}-r^{2})-(r^{\prime \prime })^{2}(1+r^{2})^{2}=0  \label{D13}
\end{equation}%
holds.
\end{theorem}

\begin{proof}
Substituting (\ref{D2}), (\ref{D3}), (\ref{D9}) and (\ref{D12}) into we get (%
\ref{D13}).
\end{proof}

\section{Chen Surfaces in $\mathbb{E}^{4}$}

Let $M$ be a smooth surface in $\mathbb{E}^{4}$ given with the patch $X(u,v)$
: $(u,v)\in D\subset \mathbb{E}^{2}.$ If we chose an orthonormal tangent
frame field $\left\{ X,Y\right\} $ 
\begin{eqnarray}
X &=&\frac{X_{u}}{\sqrt{E}},  \label{E1} \\
Y &=&\frac{\sqrt{E}}{W}\left( X_{v}-\frac{FX_{u}}{E}\right) .  \notag
\end{eqnarray}%
then the coefficients of the second fundamental form are given by%
\begin{eqnarray}
h_{11}^{\alpha } &=&\left\langle h(X,X),N_{\alpha }\right\rangle =\frac{%
c_{11}^{\alpha }}{E},\text{ }1\leq \alpha \leq 2.  \notag \\
h_{12}^{\alpha } &=&\left\langle h(X,Y),N_{\alpha }\right\rangle =\frac{1}{W}%
\left( c_{12}^{\alpha }-\frac{F}{E}c_{11}^{\alpha }\right) ,  \label{E2} \\
h_{22}^{\alpha } &=&\left\langle h(Y,Y),N_{\alpha }\right\rangle =\frac{1}{%
W^{2}}\left( Ec_{22}^{\alpha }-2Fc_{12}^{\alpha }+\frac{F^{2}}{E}%
c_{11}^{\alpha }\right) .  \notag
\end{eqnarray}

Further, the shape operator matrix of the surface $M\subset \mathbb{E}^{4}$
becomes%
\begin{equation*}
A_{N_{\alpha }}=\left( 
\begin{array}{cc}
h_{11}^{\alpha } & h_{12}^{\alpha } \\ 
h_{12}^{\alpha } & h_{22}^{\alpha }%
\end{array}%
\right) .
\end{equation*}

Hence, the mean curvature vector of a regular patch $X(u,v)$ is defined by%
\begin{eqnarray}
\overrightarrow{H} &=&\frac{1}{2}(tr(A_{N_{1}})+tr(A_{N_{2}}))  \label{E3} \\
&=&H_{1}N_{1}+H_{2}N_{2},  \notag
\end{eqnarray}%
where the functions%
\begin{equation}
H_{1}=\frac{1}{2}(h_{11}^{1}+h_{22}^{1}),\text{ }H_{2}=\frac{1}{2}%
(h_{11}^{2}+h_{22}^{2})  \label{E5}
\end{equation}%
are called the \textit{first and second harmonic curvatures \ of }$M$
respectively.

For any arbitrary orthonormal normal frame field $N_{1},N_{2}$ of $M$ \ such
that the vector field $N_{1}$ is parallel to mean curvature vector $%
\overrightarrow{H}.$ In \cite{Ch} B-Y. Chen defined the allied vector field $%
a(\overrightarrow{H})$ of the mean curvature vector field $\overrightarrow{H}
$ by the formula%
\begin{equation}
a(\overrightarrow{H})=\frac{\left\Vert \overrightarrow{H}\right\Vert }{2}%
\left\{ tr(A_{_{N_{1}}}A_{_{N_{2}}})\right\} N_{2}.  \label{E6}
\end{equation}%
In particular,the allied mean curvature vector field of the mean curvature
vector $\overrightarrow{H}$ is a well-defined normal vector field orthogonal
to $\overrightarrow{H}.$ If the allied mean vector $a(\overrightarrow{H})$
vanishes identically, then the surface $M$ is called $A$\textit{-surface} of 
$\mathbb{E}^{4}$. Furthermore, $\mathcal{A}$-surfaces are also called 
\textit{Chen surfaces} \cite{GVV1}. The class of Chen surfaces contains all
minimal and pseudo-umbilical surfaces, and also all surfaces for which $%
dimN_{1}\leq 1,$ in particular all hypersurfaces. These Chen surfaces are
said to be trivial $\mathcal{A}$-surfaces \cite{GVV2}.

\begin{theorem}
Let $M$ be a smooth surface in $\mathbb{E}^{4}$ given with the patch $X(u,v)$
: $(u,v)\in D\subset \mathbb{E}^{2}.$ Then $M$ is a non-trivial Chen
surfaces if and only if 
\begin{equation}
\left(
(h_{11}^{1})^{2}-(h_{11}^{2})^{2}+(h_{22}^{1})^{2}-(h_{22}^{2})^{2}+2(h_{12}^{1})^{2}-2(h_{12}^{2})^{2}\right) H_{1}H_{2}+
\label{E7}
\end{equation}%
\begin{equation*}
+\left(
h_{11}^{1}h_{11}^{2}+h_{22}^{1}h_{22}^{2}+2h_{12}^{1}h_{12}^{2}\right)
(H_{2}^{2}-H_{1}^{2})=0
\end{equation*}%
holds, where $H_{1}$ and $H_{2}$ are the first and second harmonic
curvatures \ of $M$ as defined before.
\end{theorem}

\begin{proof}
Suppose $M$ is a non-minimal surface in $\mathbb{E}^{4}.$ Then we
can construct another orthonormal normal frame field
\begin{equation}
\widetilde{N}_{1}=\frac{H_{1}N_{1}+H_{2}N_{2}}{\sqrt{H_{1}^{2}+H_{2}^{2}}},%
\text{ }\widetilde{N}_{2}=\frac{H_{2}N_{1}-H_{1}N_{2}}{\sqrt{%
H_{1}^{2}+H_{2}^{2}}},  \label{E8}
\end{equation}%
such that $\widetilde{N}_{1}$ is parallel to $\overrightarrow{H}.$

Furthermore, with respect to this frame we can obtain
\begin{eqnarray}
\widetilde{h}_{11}^{1} &=&\left\langle
h(X,X),\widetilde{N}_{1}\right\rangle
=\frac{H_{1}h_{11}^{1}+H_{2}h_{11}^{2}}{\sqrt{H_{1}^{2}+H_{2}^{2}}},
\notag
\\
\widetilde{h}_{12}^{1} &=&\left\langle
h(X,Y),\widetilde{N}_{1}\right\rangle
=\frac{H_{1}h_{12}^{1}+H_{2}h_{12}^{2}}{\sqrt{H_{1}^{2}+H_{2}^{2}}},
\notag
\\
\widetilde{h}_{22}^{1} &=&\left\langle
h(Y,Y),\widetilde{N}_{1}\right\rangle
=\frac{H_{1}h_{22}^{1}+H_{2}h_{22}^{2}}{\sqrt{H_{1}^{2}+H_{2}^{2}}},
\label{E9} \\
\widetilde{h}_{11}^{2} &=&\left\langle
h(X,X),\widetilde{N}_{2}\right\rangle
=\frac{H_{2}h_{11}^{1}-H_{1}h_{11}^{2}}{\sqrt{H_{1}^{2}+H_{2}^{2}}},
\notag
\\
\widetilde{h}_{12}^{2} &=&\left\langle
h(X,Y),\widetilde{N}_{2}\right\rangle
=\frac{H_{2}h_{12}^{1}-H_{1}h_{12}^{2}}{\sqrt{H_{1}^{2}+H_{2}^{2}}},
\notag
\\
\widetilde{h}_{22}^{2} &=&\left\langle
h(Y,Y),\widetilde{N}_{2}\right\rangle
=\frac{H_{2}h_{22}^{1}-H_{1}h_{22}^{2}}{\sqrt{H_{1}^{2}+H_{2}^{2}}}.
\notag
\end{eqnarray}

So, the shape operator matrices of $M$ with respect to
$\widetilde{N}_{1}$ and $\widetilde{N}_{2}$ become
\begin{equation}
A_{\widetilde{N}_{1}}=\left(
\begin{array}{cc}
\frac{H_{1}h_{11}^{1}+H_{2}h_{11}^{2}}{\sqrt{H_{1}^{2}+H_{2}^{2}}} & \frac{%
H_{1}h_{12}^{1}+H_{2}h_{12}^{2}}{\sqrt{H_{1}^{2}+H_{2}^{2}}} \\
\frac{H_{1}h_{12}^{1}+H_{2}h_{12}^{2}}{\sqrt{H_{1}^{2}+H_{2}^{2}}} & \frac{%
H_{1}h_{22}^{1}+H_{2}h_{22}^{2}}{\sqrt{H_{1}^{2}+H_{2}^{2}}}%
\end{array}%
\right) ,\text{ }A_{\widetilde{N}_{2}}=\left(
\begin{array}{cc}
\frac{H_{2}h_{11}^{1}-H_{1}h_{11}^{2}}{\sqrt{H_{1}^{2}+H_{2}^{2}}} & \frac{%
H_{2}h_{12}^{1}-H_{1}h_{12}^{2}}{\sqrt{H_{1}^{2}+H_{2}^{2}}} \\
\frac{H_{2}h_{12}^{1}-H_{1}h_{12}^{2}}{\sqrt{H_{1}^{2}+H_{2}^{2}}} & \frac{%
H_{2}h_{22}^{1}-H_{1}h_{22}^{2}}{\sqrt{H_{1}^{2}+H_{2}^{2}}}%
\end{array}%
\right) ,  \label{E10}
\end{equation}%
respectively.

Suppose $M$ is a non-trivial Chen surface then by definition $tr\left( A_{%
\widetilde{N}_{1}}A_{\widetilde{N}_{2}}\right) =0.$ So by the use of (\ref%
{E10}) we get the result.

Conversely, if the equality (\ref{E7}) holds then $tr\left( A_{\widetilde{N}%
_{1}}A_{\widetilde{N}_{2}}\right) =0$. So, $M$ is a non-trivial Chen
surface.
\end{proof}

We obtain the following result.

\begin{theorem}
Let $M$ be an Aminov surface in $\mathbb{E}^{4}$ given with the Monge patch (%
\ref{D1}). Then $M$ is a non-trivial Chen surface.
\end{theorem}

\begin{proof}
Suppose $M$ is an Aminov surface in $\mathbb{E}^{4}$ given with the
parametrization (\ref{D1}). By the use of (\ref{C7}) with (\ref{E2})
a
simple calculation gives%
\begin{eqnarray}
h_{11}^{1} &=&\frac{r^{\prime \prime }(u)\cos v}{\varphi \psi ^{2}},\text{\ }%
h_{12}^{1}=\frac{-r^{\prime }(u)\sin v}{\varphi \psi \omega },  \notag \\
h_{22}^{1} &=&\frac{-r(u)\cos v}{\varphi \omega ^{2}},\text{ }h_{11}^{2}=%
\frac{\omega r^{\prime \prime }(u)\sin v}{\varphi \psi ^{3}},  \label{E11} \\
h_{12}^{2} &=&\frac{r^{\prime }(u)\cos v}{\varphi \omega ^{2}},\text{ }%
h_{22}^{2}=\frac{-r(u)\sin v}{\varphi \psi \omega },  \notag
\end{eqnarray}%
where $\varphi $, $\psi $ and $\omega $ are differentiable functions
defined by
\begin{eqnarray}
\varphi &=&\sqrt{1+(r^{\prime }(u))^{2}\cos ^{2}v+(r(u))^{2}\sin
^{2}v},
\notag \\
\psi &=&\sqrt{1+(r^{\prime }(u))^{2}},  \label{E12} \\
\omega &=&\sqrt{1+(r(u))^{2}}.  \notag
\end{eqnarray}

Substituting (\ref{E11}) into (\ref{E7}) we get the result.
\end{proof}

\textbf{Acknowledgements.} We thank Prof. Dr. Yuriy Aminov of National
Academy of Sciences of Ukraine for useful comments and for his encouragement
during the study.

\end{document}